\documentclass[3p,fleqn]{elsarticle}
\usepackage{amssymb}
\usepackage{amsfonts}
\usepackage{amsmath}
\usepackage{amsthm}
\usepackage{color}
\usepackage{graphicx}
\usepackage{epsfig,mathrsfs}
\usepackage[bf,SL,BF]{subfigure}
\usepackage{fancyhdr}
\usepackage{CJK}
\usepackage{caption}
\usepackage{wrapfig}
\usepackage{cases}
\usepackage{multirow}
\usepackage{algorithm}
\usepackage{algorithmic}

\newtheorem{theorem}{Theorem}
\newtheorem{lemma}[theorem]{Lemma}
\newtheorem{corollary}[theorem]{Corollary}

\newtheorem{method}[theorem]{Method}
\numberwithin{equation}{section}

\numberwithin{theorem}{section}
\journal{}

\begin{document}

\begin{frontmatter}
\title{On semi-convergence of generalized skew-Hermitian triangular splitting iteration methods for singular saddle-point problems}
\author{Yan Dou}
\author{Ai-Li Yang\corref{yal}}\ead{yangaili@lzu.edu.cn; cmalyang@gmail.com}%\fnref{yal}
\cortext[yal]{Corresponding author. Tel.: +86 931 8912483; fax: +86 931 8912481.}
\author{Yu-Jiang Wu}

\address{School of Mathematics and Statistics, Lanzhou University, Lanzhou 730000, PR China}
%\address[fup]{Department of Mathematics, Federal University of Paran\'{a}, Centro Polit\'{e}cnico, CP: 19.081, 81531-980, Curitiba, PR
%Brazil}

\begin{abstract}
Recently, Krukier et al. [Generalized skew-Hermitian triangular splitting iteration methods for saddle-point linear systems, Numer. Linear Algebra Appl. 21 (2014) 152-170] proposed an efficient \emph{generalized\ skew-Hermitian\ triangular\ splitting} (GSTS) iteration method for nonsingular saddle-point linear systems with strong skew-Hermitian parts. In this work, we further use the GSTS method to solve \emph{singular} saddle-point problems. The semi-convergence properties of GSTS method are analyzed by using singular value decomposition and Moore-Penrose inverse, under suitable restrictions on the involved iteration parameters. Numerical results are presented to demonstrate the feasibility and efficiency of the GSTS iteration methods, both used as solvers and preconditioners for GMRES method.

\noindent\emph{MSC:} 65F08; 65F10; 65F20
\end{abstract}
\begin{keyword}
singular saddle-point problems; skew-Hermitian triangular splitting; iteration method; semi-convergence; Moore-Penrose inverse; singular value decomposition
\end{keyword}

\end{frontmatter}
\section{Introduction}
Consider the following saddle-point linear system:
\begin{equation}\label{01}
\mathcal{A}\,u\equiv\left(
                    \begin{array}{cc}
                      M & E \\
                      -E^{*} & 0 \\
                    \end{array}
                  \right)\left(
                            \begin{array}{c}
                              u_{1} \\
                              u_{2} \\
                            \end{array}
                          \right)=\left(
                                        \begin{array}{c}
                                          f_{1} \\
                                          f_{2} \\
                                        \end{array}
                                      \right)\equiv f,
\end{equation}
where $M\in \mathbb{C}^{p\times p}$ is a Hermitian positive definite matrix, $E\in \mathbb{C}^{p\times q}$ is a rectangular matrix satisfying $q\leq p$, and $f\in \mathbb{C}^{p+q}$ is a given vector in the range of $\mathcal{A}\in \mathbb{C}^{(p+q)\times (p+q)}$, with $f_{1}\in \mathbb{C}^{p}$ and $f_{2}\in \mathbb{C}^{q}$. This kind of linear systems arise in a variety of scientific and engineering applications, such as computational fluid dynamics, constrained optimization, optimal control, weighted least-squares problems, electronic networks, computer graphic etc, and typically result from mixed or hybrid finite element approximation of second-order elliptic problems or the Stokes equations; see \cite{BPW20051,SSY19981,WSY2004139,ZBY2009808,Brezzi1991}.

When matrix $E$ is of full column rank, the saddle-point matrix $\mathcal{A}$ is nonsingular. A number of effective iteration methods, such as matrix splitting iteration methods, Minimum residual methods, Krylov subspace iteration methods etc, have been proposed in the literature to approximate the unique solution of the nonsingular saddle-point problems \eqref{01}; see \cite{BGP20041,BGL20051,EG19941645,FRSW1998527,KGW20001300,B2009447,BW20082900} and the references therein. Recently, Krukier et al. \cite{Krukier2013} proposed a generalized skew-Hermitian triangular splitting (GSTS) iteration method for solving the linear systems with strong skew-Hermitian parts. When used for approximating the solution of the nonsingular saddle-point problem \eqref{01}, the GSTS method can be described as follows.
\begin{method}\label{met:1}(The GSTS iteration method) Given initial guesses $u^{(0)}_{1}\in \mathbb{C}^{p}$ and $u^{(0)}_{2}\in \mathbb{C}^{q}$, for $k=0,1,2\ldots$, until $u^{(k)}=[u_1^{(k)};u_2^{(k)}]$ convergence
\begin{enumerate}
  \item[(i)] compute $u_2^{(k+1)}$ from
  \begin{equation}\label{02}
    u_2^{(k+1)}=u_2^{(k)}+\tau B^{-1}\left[\omega_1E^*M^{-1}\left(f_1-Eu_2^{(k)}\right)+(1-\omega_1)E^*u_1^{(k)}+f_2\right];
  \end{equation}
  \item[(ii)] compute $u_2^{(k+1)}$ from
  \begin{equation}\label{19}
    u_1^{(k+1)}=(1-\tau) u_1^{(k)}+ M^{-1}\left[E\left((\omega_2-\tau)u_2^{(k)}-\omega_2u_2^{(k+1)}\right)+\tau f_1\right],
  \end{equation}
\end{enumerate}
where $\omega_{1}$ and $\omega_{2}$ are two nonnegative acceleration parameters with at least one of them being nonzero, $\tau$ is a positive parameter,
$B\in\mathbb{C}^{q\times q}$ is a Hermitian positive definite matrix, which is chosen as an approximation of the Shur complement $S_M:=E^{*}M^{-1}E$.
\end{method}
% used to approximate the

Theoretical analysis and numerical experiments in \cite{Krukier2013} have shown that the GSTS iteration method is convergent under suitable restrictions on iteration parameters. Moreover, no matter as a solver or as a preconditioner for GMRES method, the GSTS method is robust and effective for solving the large sparse nonsingular saddle-point linear systems. However, matrix $E$ in saddle-point matrix $\mathcal{A}$ is rank deficient in many real world applications, such as the discretization of incompressible steady state Stokes problem with suitable boundary conditions; see \cite{WSY2004139,ZhangWei2010139}. In this case, the saddle-point linear systems are always singular and consistent. The Uzawa algorithm and its variants \cite{ZBY2009808,ZLW2014334}, Hermitian and skew-Hermitian splitting iteration method \cite{Bai2010171,BGN2002603,Li20122338}, general stationary linear iteration method \cite{Cao20081382,ZhangWei2010139}, Krylov subspace methods (preconditioned by block-diagonal, block-tridiagonal or constraint preconditioners) \cite{WSY2004139,BW199737,ZhangShen2013116} etc, can be used to approximate a solution of the singular and consistent saddle-point linear system.

In this work, owing to the high efficiency of the GSTS iteration method used for solving the nonsingular saddle-point linear systems, we will further analyze the feasibility and efficiency of the GSTS iteration method when it is used for solving the singular saddle-point problems \eqref{01} with Hermitian positive definite matrix $M\in \mathbb{C}^{p\times p}$ and rank deficient matrix $E\in \mathbb{C}^{p\times q}$. Since matrix $E$ is rank deficient, the Shur complement $S_M=E^{*}M^{-1}E$ is Hermitian positive semi-definite. As the approximation of Shur complement, it may be better if we choose matrix $B$ being a Hermitian positive semi-definite matrix and having the same null space with Shur complement $S_M$. In this way, matrix $B$ is singular, we replace iteration scheme \eqref{02} in Method \ref{met:1} by the scheme of the form
\begin{equation}\label{14}
  u_2^{(k+1)}=u_2^{(k)}+\tau B^{\dag}\left[\omega_1E^*M^{-1}\left(f_1-Eu_2^{(k)}\right)+(1-\omega_1)E^*u_1^{(k)}+f_2\right],
\end{equation}
where $B^{\dag}$ is the Moore-Penrose inverse \cite{BermanPlemmons1994,Kucera2011} of the singular matrix $B$, which satisfies
\[
B=BB^{\dag}B,\quad B^{\dag}=B^{\dag}BB^{\dag},\quad BB^{\dag}=(BB^{\dag})^{*}, \quad B^{\dag}B=(B^{\dag}B)^{*}.
\]
The convergence properties of the GSTS iteration methods, with Hermitian positive definite and singular Hermitian positive semi-definite matrices $B$, will be carefully analyzed. Moreover, the feasibility and efficiency of the GSTS iteration methods for singular and consistent saddle-point problems will also be numerically verified.

The remainder part of this work is organized as follows. In Section \ref{sec2} we give the semi-convergence concepts of the GSTS iteration methods with different choices of matrix $B$, i.e., $B$ is Hermitian positive definite and singular Hermitian positive semi-definite. When $B$ is Hermitian positive definite, the semi-convergence properties of the GSTS iteration method are analyzed in Section \ref{sec3}. In Section \ref{sec4}, we give the semi-convergence properties of the GSTS method with $B$ being singular Hermitian positive semi-definite. In Section \ref{sec5}, numerical results are presented to show the feasibility and effectiveness of the GSTS iteration methods for solving the singular saddle-point linear systems. Finally, in Section \ref{sec6}, we end this work with a brief conclusion.

\section{Basic concepts and lemmas}\label{sec2}
We split matrix $\mathcal{A}$ into its Hermitian and skew-Hermitian parts, i.e., $\mathcal{A}=\mathcal{A}_{H}+\mathcal{A}_{S}$, where
\begin{equation}\label{32}
\mathcal{A}_{H}=\frac{1}{2}(\mathcal{A}+\mathcal{A}^{*})=\left(
                                              \begin{array}{cc}
                                                M & 0 \\
                                                0 & 0 \\
                                              \end{array}
                                            \right),\quad
\mathcal{A}_{S}=\frac{1}{2}(\mathcal{A}-\mathcal{A}^{*})=\left(
                                                      \begin{array}{cc}
                                                        0 & E \\
                                                        -E^{*} & 0 \\
                                                      \end{array}
                                                    \right).
\end{equation}
Let $\mathcal{K}_{L}$ and $\mathcal{K}_{U}$ be, respectively, the strictly lower-triangular and the strictly upper-triangular parts of $\mathcal{A}_{S}$ satisfying
\begin{equation}\label{08}
\mathcal{A}_{S}=\mathcal{K}_{L}+\mathcal{K}_{U}=\left(
                                              \begin{array}{cc}
                                                0 & 0 \\
                                                -E^{*} & 0 \\
                                              \end{array}
                                            \right)+\left(
                                                      \begin{array}{cc}
                                                        0 & E \\
                                                        0 & 0 \\
                                                      \end{array}
                                                    \right),
\end{equation}
and denote
\begin{equation}\label{21}
\mathcal{B}_{c}=\left(
                  \begin{array}{cc}
                    M & 0 \\
                    0 & B \\
                  \end{array}
                \right).
\end{equation}

In the following two subsections, we give some basic concepts and useful lemmas for the analysis of the semi-convergence properties of the GSTS iteration methods according to the choices of matrix $B$.
\subsection{Matrix $B$ is Hermitian positive definite}
Firstly, we consider the case that matrix $B$ used in Method \ref{met:1} is Hermitian positive definite. Combining iteration schemes \eqref{02} and \eqref{19}, the GSTS iteration method can be rewritten as
\begin{equation}\label{20}
u^{(k+1)}=u^{(k)}-\tau\mathcal{B}(\omega_{1},\omega_{2})^{-1}(\mathcal{A}u^{(k)}-f),
\end{equation}
where
\begin{equation}\label{25}
\mathcal{B}(\omega_{1},\omega_{2})=(\mathcal{B}_{c}+\omega_{1}\mathcal{K}_{L})\mathcal{B}_{c}^{-1}
(\mathcal{B}_{c}+\omega_{2}\mathcal{K}_{U}).
\end{equation}
The iteration matrix is
\begin{equation}\label{13}
\mathcal{G}(\omega_{1},\omega_{2},\tau)=I-\tau\mathcal{B}(\omega_{1},\omega_{2})^{-1}\mathcal{A}.
\end{equation}

Iteration scheme \eqref{20} can be induced from the splitting
\begin{equation*}
\mathcal{A}=\mathcal{M}(\omega_{1},\omega_{2},\tau)-\mathcal{N}(\omega_{1},\omega_{2},\tau),
\end{equation*}
where
\[
\mathcal{M}(\omega_{1},\omega_{2},\tau)=(1/\tau)\mathcal{B}(\omega_{1},\omega_{2}), \quad\mathcal{N}(\omega_{1},\omega_{2},\tau)=(1/\tau) \left(\mathcal{B}(\omega_{1},\omega_{2})-\tau\mathcal{A}\right).\]
Hence, matrix $\mathcal{M}(\omega_{1},\omega_{2},\tau)$, or $\mathcal{B}(\omega_{1},\omega_{2})$, can be viewed as a preconditioner for the saddle-point linear system \eqref{01}, which may be used to accelerate the convergence rate of Krylov subspace methods, such as the generalized minimum residual (GMRES) method and the quasi-minimal residual (QMR) method.

For the semi-convergence of iteration scheme \eqref{20}, we give the following useful lemma.
\begin{lemma}\cite{BermanPlemmons1994}\label{lem:6}
The iterative scheme
\[u^{(k+1)}=u^{(k)}-\mathcal{M}^{-1}(\mathcal{A}u^{(k)}-f)\]
is semi-convergent, if and only if its iteration matrix $\mathcal{G}=I-\mathcal{M}^{-1}\mathcal{A}$ satisfies
  \begin{enumerate}
    \item[(1)] The pseudo-spectral radius of matrix $\mathcal{G}$ is less than $1$, i.e., \[\gamma(\mathcal{G}):=\max \{|\lambda|:\lambda\in\sigma(\mathcal{G})\setminus{1}\}<1,\] where $\sigma(\mathcal{G})$ is the set of eigenvalues of matrix $\mathcal{G}$;
    \item[(2)] $\text{index}(I-\mathcal{G})=1$, or equivalently, $\text{rank}(I-\mathcal{G})=\text{rank}((I-\mathcal{G})^{2})$.
  \end{enumerate}
\end{lemma}

\subsection{Matrix $B$ is singular and Hermitian positive semi-definite}

When matrix $E$ in \eqref{01} is rank deficient, the Shur complement $S_M=E^{*}M^{-1}E$ is singular and Hermitian positive semi-definite. As the approximation of $S_M$, matrix $B$ is chosen as $E^{*}P^{-1}E$, where $P$ is an approximation of $M$ and is Hermitian positive definite. Hence, matrix $B$ is singular and Hermitian positive semi-definite, and has the same null space with Shur complement $S_M$.

Owing to the singularity of matrix $B$, we replace iteration scheme \eqref{02} by \eqref{14} and obtain a more generalized GSTS iteration method. Based on \eqref{08} and \eqref{21}, the GSTS iteration method can be rewritten as
\begin{equation}\label{22}
u^{(k+1)}=u^{(k)}-\tau\mathcal{B}(\omega_{1},\omega_{2})^{\dag}(\mathcal{A}u^{(k)}-f),
\end{equation}
where
\begin{equation}
\mathcal{B}(\omega_{1},\omega_{2})=(\mathcal{B}_{c}+\omega_{1}\mathcal{K}_{L})\mathcal{B}_{c}^{\dag}
(\mathcal{B}_{c}+\omega_{2}\mathcal{K}_{U}).
\end{equation}
Iteration matrix is
\begin{equation}\label{23}
\mathcal{G}(\omega_{1},\omega_{2},\tau)=I-\tau\mathcal{B}(\omega_{1},\omega_{2})^{\dag}\mathcal{A}.
\end{equation}
Here, matrix $\mathcal{B}(\omega_{1},\omega_{2})$ can also be viewed as a preconditioner for singular saddle-point linear system \eqref{01}. The difference is that the preconditioner  $\mathcal{B}(\omega_{1},\omega_{2})$ introduced in this subsection is singular.

Comparing with iteration scheme \eqref{20}, we need one more condition to keep the semi-convergence of iteration scheme \eqref{22} since matrix $B$ is singular.
\begin{lemma}\label{lem:1}\cite{Cao20081382}
The iterative scheme
\[
u^{(k+1)}=u^{(k)}-\mathcal{M}^{\dag}(\mathcal{A}u^{(k)}-f)
\]
is semi-convergent if and only if the following three conditions are fulfilled:
  \begin{enumerate}
    \item[(1)] The pseudo-spectral radius of matrix $\mathcal{G}$ is less than $1$, i.e., $\gamma(\mathcal{G})<1$, where $\mathcal{G}\equiv I-\mathcal{M}^{\dag}\mathcal{A}$ is the iteration matrix;
    \item[(2)] \text{null}$(\mathcal{M}^{\dag}\mathcal{A})$=\text{null}$(\mathcal{A})$;
    \item[(3)] index$(I-\mathcal{G})$=1, or equivalently, rank$(I-\mathcal{G})$=rank$((I-\mathcal{G})^{2})$.
  \end{enumerate}
\end{lemma}

\section{The semi-convergence of GSTS method with $B$ being Hermitian positive definite}\label{sec3}
Since matrix $B$ is nonsingular, it is easy to see that matrix $\mathcal{B}(\omega_{1},\omega_{2})$ is invertible. The inverse matrix of $\mathcal{B}(\omega_{1},\omega_{2})$ has the following explicit form
\begin{equation*}
\mathcal{B}(\omega_{1},\omega_{2})^{-1}=\left(
                                            \begin{array}{cc}
                                              M^{-1}-\omega_{1}\omega_{2}M^{-1}EB^{-1}E^{*}M^{-1} & -\omega_{2}M^{-1}EB^{-1} \\
                                              \omega_{1}B^{-1}E^{*}M^{-1} & B^{-1} \\
                                            \end{array}
                                          \right).
\end{equation*}
The iteration matrix $\mathcal{G}(\omega_{1},\omega_{2},\tau)$ can be written as
\begin{equation*}
\mathcal{G}(\omega_{1},\omega_{2},\tau)=\left(
     \begin{array}{cc}
       (1-\tau)I_{p}-\tau\omega_{2}(1-\omega_{1})M^{-1}EB^{-1}E^{*} & -\tau M^{-1}E(I_q-\omega_{1}\omega_{2}B^{-1}E^{*}M^{-1}E) \\
                                              \tau(1-\omega_{1})B^{-1}E^{*} & I_{q}-\tau\omega_{1}B^{-1}E^{*}M^{-1}E \\
     \end{array}
   \right).
\end{equation*}

%\begin{lemma}\cite{ZBY2009808}\label{lem:5}
%Let $H\in \mathbb{C}^{p\times p}$ be any square matrix, and $I\in \mathbb{C}^{q\times q}$ be the identity matrix, with $p$ and $q$ being two positive integers. Then the partitioned matrix
%\begin{equation*}
%G=\left(
%    \begin{array}{cc}
%      H & 0 \\
%      L & I \\
%    \end{array}
%  \right)
%\end{equation*}
%is semi-convergent if and only if either of the following conditions holds true:
%  \begin{enumerate}
%    \item $L=0$ and $H$ is semi-convergent;
%    \item $\rho(H)<1$.
%  \end{enumerate}
%\end{lemma}
In the following, we further study the semi-convergence properties of GSTS iteration method in which matrix $B$ is Hermitian positive definite. In fact, we only need to verify the two conditions presented in Lemma \ref{lem:6}.
\subsection{The conditions for index$(I-\mathcal{G}(\omega_{1},\omega_{2},\tau))=1$}
\begin{lemma}
Let matrices $\mathcal{A}$ and $\mathcal{B}(\omega_{1},\omega_{2})$ be defined by \eqref{01} and \eqref{25}, respectively. Then, we have index$(I-\mathcal{G}(\omega_{1},\omega_{2},\tau))=1$, or equivalently,
\begin{equation}\label{31}
  \text{rank}(I-\mathcal{G}(\omega_{1},\omega_{2},\tau))=\text{rank}((I-\mathcal{G}(\omega_{1},\omega_{2},\tau))^{2}).
\end{equation}
\end{lemma}
\begin{proof}
Inasmuch as $\mathcal{G}(\omega_{1},\omega_{2},\tau)=I-\tau\mathcal{B}(\omega_{1},\omega_{2})^{-1}\mathcal{A}$, equality \eqref{31} holds if
\[
\text{null}((\mathcal{B}(\omega_{1},\omega_{2})^{-1}\mathcal{A})^{2})=\text{null}(\mathcal{B}(\omega_{1},\omega_{2})^{-1}\mathcal{A}).
\]
It is obvious that $\text{null}((\mathcal{B}(\omega_{1},\omega_{2})^{-1}\mathcal{A})^{2})\supseteq \text{null}(\mathcal{B}(\omega_{1},\omega_{2})^{-1}\mathcal{A})$, we only need to prove
\begin{equation}\label{27}
\text{null}((\mathcal{B}(\omega_{1},\omega_{2})^{-1}\mathcal{A})^{2})\subseteq \text{null}(\mathcal{B}(\omega_{1},\omega_{2})^{-1}\mathcal{A}).
\end{equation}

Let $x=(x_{1}^{*},x_{2}^{*})^{*}\in\mathbb{C}^{p+q}$ satisfy $(\mathcal{B}(\omega_{1},\omega_{2})^{-1}\mathcal{A})^{2}x=0$. Denote $y=\mathcal{B}(\omega_{1},\omega_{2})^{-1}\mathcal{A}x$, then simple calculation gives
\begin{equation}\label{26}
y=\left(
     \begin{array}{c}
       y_{1} \\
       y_{2} \\
     \end{array}
   \right)=\left(
               \begin{array}{cc}
                (I_{p}+\omega_{2}(1-\omega_{2})M^{-1}EB^{-1}E^{*})x_{1}+M^{-1}E(I_q-\omega_{1}\omega_{2}B^{-1}E^{*}M^{-1}E)x_{2} \\
                B^{-1}E^{*}\left((\omega_{1}-1)x_{1}+\omega_{1}M^{-1}Ex_{2}\right)
                \end{array}\right).
\end{equation}
In the following, we only need to prove $y=0$.
From $\mathcal{B}(\omega_{1},\omega_{2})^{-1}\mathcal{A}y=(\mathcal{B}(\omega_{1},\omega_{2})^{-1}\mathcal{A})^{2}x=0$, we have $\mathcal{A}y=0$, i.e.,
\begin{equation}\label{24}
My_{1}+Ey_{2}=0 \quad\text{and}\quad -E^{*}y_{1}=0.
\end{equation}
Note that $M$ is nonsingular, solving $y_{1}$ from the first equality of \eqref{24} and taking it into the second equality, it follows that $E^{*}M^{-1}Ey_{2}=0$. Hence,
\begin{equation*}
(Ey_{2})^{*}M^{-1}(Ey_{2})=y_{2}(E^{*}M^{-1}Ey_{2})=0.
\end{equation*}
Owing to the Hermitian positive definiteness of matrix $M^{-1}$, we can obtain that $Ey_{2}=0$. Taking it into the first equality of \eqref{24} gives $y_{1}=0$. Furthermore, using $Ey_{2}=0$ and \eqref{26}, we have
\[
Ey_{2}=EB^{-1}E^{*}\left((\omega_{1}-1)x_{1}+\omega_{1}M^{-1}Ex_{2}\right)=0.
\]
Since matrix $B$ is Hermitian positive definite, we can derive, with similar technique, that
\[E^{*}\left((\omega_{1}-1)x_{1}+\omega_{1}M^{-1}Ex_{2}\right)=0,\]
which means
\[y_2=B^{-1}E^{*}\left((\omega_{1}-1)x_{1}+\omega_{1}M^{-1}Ex_{2}\right)=0.\]
Thus, $\mathcal{B}(\omega_{1},\omega_{2})^{-1}\mathcal{A}x=y=0$, i.e., the inclusion relation \eqref{27} holds.
\end{proof}

\subsection{The conditions for $\gamma(\mathcal{G}(\omega_{1},\omega_{2},\tau))<1$}
Assume that the column rank of $E$ is $r$, i.e., $r=\text{rank}(E)$. Let
\begin{equation}\label{10}
  E=U(E_{r},0)V^{*}
\end{equation}
be the singular value decomposition of $E$, where $U\in\mathbb{C}^{p\times p}$ and $V\in\mathbb{C}^{q\times q}$ are two unitary matrices,
$E_{r}=(\Sigma_{r},0)^*\in\mathbb{C}^{p\times r}$
%E_{r}=\left( \begin{array}{c}
%                \Sigma_{r}\\  0
%              \end{array}
%            \right)\in\mathbb{C}^{p\times r}
%\]
and $\Sigma_{r}=\text{diag}(\sigma_{1}, \sigma_{2}, \cdots, \sigma_{r})$, with $\sigma_{i}$ being the singular value of matrix $E$.

%,\quad \Sigma_{r}=\left( \begin{array}{ccc} \sigma_1 &&\\ &\ddots&\\ &&\sigma_r  \end{array} \right)

We partition matrix $V$ as $V=(V_{1},V_{2})$ with $V_{1}\in\mathbb{C}^{q\times r}$, $V_{2}\in\mathbb{C}^{q\times (q-r)}$ and define
\begin{equation}\label{17}
\mathcal{P}=\left(
              \begin{array}{cc}
                U & 0 \\
                0 & V \\
              \end{array}
            \right).
\end{equation}
It is obvious that $\mathcal{P}$ is a $(p+q)\times(p+q)$ unitary matrix, and the iteration matrix $\mathcal{G}(\omega_{1},\omega_{2},\tau)$ is unitarily similar to the matrix $\hat{\mathcal{G}}(\omega_{1},\omega_{2},\tau)=\mathcal{P}^{*}\mathcal{G}(\omega_{1},\omega_{2},\tau)\mathcal{P}$. Hence, the pseudo-spectral radii of matrices $\hat{\mathcal{G}}(\omega_{1},\omega_{2},\tau)$ and $\mathcal{G}(\omega_{1},\omega_{2},\tau)$ are same, we in the following only need to analyze the pseudo-spectral radius of matrix $\hat{\mathcal{G}}(\omega_{1},\omega_{2},\tau)$.

Denoting $\hat{M}=U^{*}MU$ and $\hat{B}=V^{*}BV$, we have
\begin{equation}\label{28}
\hat{B}^{-1}=\left(\begin{array}{cc}
\hat{B}^{-1}_{11} & \hat{B}^{-1}_{12} \\
\hat{B}^{-1}_{21} & \hat{B}^{-1}_{22} \\\end{array}
\right)=\left(\begin{array}{cc}
V^{*}_{1}B^{-1}V_{1} & V^{*}_{1}B^{-1}V_{2} \\
V^{*}_{2}B^{-1}V_{1} & V^{*}_{2}B^{-1}V_{2} \\
\end{array}\right).
\end{equation}
Furthermore, we can derive that
\begin{equation*}
\hat{\mathcal{G}}(\omega_{1},\omega_{2},\tau)=\left(
     \begin{array}{cc}
       \hat{\mathcal{G}}_{1}(\omega_{1},\omega_{2},\tau) & 0 \\
       \hat{\mathcal{L}}(\omega_{1},\omega_{2},\tau) & I_{q-r} \\
     \end{array}
   \right),
\end{equation*}
where
\begin{equation*}
\hat{\mathcal{G}}_{1}(\omega_{1},\omega_{2},\tau)=\left(
                                                    \begin{array}{cc}
       (1-\tau)I_{p}-\tau\omega_{2}(1-\omega_{1})\hat{M}^{-1}E_{r}\hat{B}^{-1}_{11}E^{*}_{r} & -\tau \hat{M}^{-1}E_{r}(I_q-\omega_{1}\omega_{2}\hat{B}^{-1}_{11}E^{*}_{r}\hat{M}^{-1}E_{r}) \\
                                              \tau(1-\omega_{1})\hat{B}^{-1}_{11}E^{*}_{r} & I_{q}-\tau\omega_{1}\hat{B}^{-1}_{11}E^{*}_{r}\hat{M}^{-1}E_{r} \\
                                                    \end{array}
                                                  \right)
\end{equation*}
and
\begin{equation*}
\hat{\mathcal{L}}(\omega_{1},\omega_{2},\tau)=\left(
                                                \begin{array}{cc}
                                                 \tau(1-\omega_{1})\hat{B}^{-1}_{21}E^{*}_{r} & -\tau\omega_{1}\hat{B}^{-1}_{21}E^{*}_{r}\hat{M}^{-1}E_{r} \\
                                                \end{array}
                                              \right).
\end{equation*}
Then, $\gamma(\hat{\mathcal{G}}(\omega_{1},\omega_{2},\tau))<1$ holds if we have $\rho(\hat{\mathcal{G}}_{1}(\omega_{1},\omega_{2},\tau))<1$.

Note that $\hat{\mathcal{G}}_{1}(\omega_{1},\omega_{2},\tau)$ is the iteration matrix of GSTS iteration method applied to the nonsingular saddle-point problem
\begin{equation}\label{15}
\hat{\mathcal{A}}\hat{u}:=\left(
  \begin{array}{cc}
    \hat{M} & E_{r} \\
    -E^{*}_{r} & 0 \\
  \end{array}
\right)\left(
         \begin{array}{c}
           \hat{u}_{1} \\
           \hat{u}_{2} \\
         \end{array}
       \right)=\left(
                 \begin{array}{c}
                   \hat{f}_{1} \\
                   \hat{f}_{2} \\
                 \end{array}
               \right)=:\hat{f}.
\end{equation}
Moreover, in the iteration process, we have
\begin{equation}\label{16}
\hat{\mathcal{B}}(\omega_{1},\omega_{2})=(\hat{\mathcal{B}}_{c}+\omega_{1}\hat{\mathcal{K}}_{L})\hat{\mathcal{B}}_{c}^{-1}(\hat{\mathcal{B}}_{c}+\omega_{2}\hat{\mathcal{K}}_{U})\quad \text{and}\quad\hat{\mathcal{B}}_{c}=\left(
                        \begin{array}{cc}
                          \hat{M} & 0 \\
                          0 & \hat{B}_{11} \\
                        \end{array}
                      \right),
\end{equation}
where $\hat{B}_{11}\in \mathbb{C}^{r\times r}$ defined in \eqref{28} is Hermitian positive definite, and $\hat{\mathcal{K}}_{L}$ and $\hat{\mathcal{K}}_{U}$ are the strictly lower-triangular and the strictly upper-triangular parts of $\hat{\mathcal{A}}_{S}=(1/2)(\hat{\mathcal{A}}-\hat{\mathcal{A}}^{*})$, respectively.

For convenience, we denote by
\begin{equation*}
\alpha:=\frac{z^{*}E_{r}^{*}\hat{M}^{-1}E_{r}z}{z^{*}z}\quad \text{and} \quad \beta_1:=\frac{z^{*}\hat{B}_{11}z}{z^{*}z}.
\end{equation*}
%Thereout, using Theorem 3.2 in \cite{Krukier2013}, we have the property about the eigenvalues of the matrix $\hat{\mathcal{G}}_{1}(\omega_{1},\omega_{2},\tau)$.
%\begin{theorem}
%Let the matrices $\hat{\mathcal{A}}$ and $\hat{\mathcal{B}}(\omega_{1},\omega_{2})$ be defined by \eqref{15} and \eqref{16}, respectively. Assume that $\mu$ is an eigenvalue of the iteration matrix $\hat{\mathcal{G}}_{1}(\omega_{1},\omega_{2},\tau)$ and $(y^{*},z^{*})^{*}\in \mathbb{C}^{p+r}$ is the corresponding eigenvector with $y\in \mathbb{C}^{p}$ and $z\in \mathbb{C}^{r}$. Then $\mu$ satisfies the quadratic equation
%\begin{equation}\label{18}
%\beta_{1}\mu^{2}+[\tau(\beta_{1}+\alpha-\tilde{\omega}\alpha)-2\beta_{1}]\mu+\tau^{2}\alpha-\tau(\beta_{1}+\alpha-\tilde{\omega}\alpha)+\beta_{1}=0,
%\end{equation}
%where $\tilde{\omega}=(\omega_{1}-1)(\omega_{2}-1)$.
%\end{theorem}
By making use of Theorem 3.3 in \cite{Krukier2013}, we derive the following result.
\begin{lemma}\label{lem:7}
Denote $\tilde{\omega}=(\omega_{1}-1)(\omega_{2}-1)$. Let matrices $\mathcal{A}$ and $\mathcal{B}(\omega_{1},\omega_{2})$ be defined by \eqref{01} and \eqref{25}, respectively. Then $\gamma(\mathcal{G}(\omega_{1},\omega_{2},\tau))<1$ holds, provided that the parameters $\omega_{1}$, $\omega_2$ satisfy
\[
\tilde{\omega}<\frac{\alpha+\beta_1}{\alpha},
\]
and the parameter $\tau$ satisfies
\begin{enumerate}
  \item[(a)] if $[\beta_1+(1-\tilde{\omega})\alpha]^{2}-4\alpha\beta_1\leq 0$, then
  \[
  0<\tau<\frac{\beta_1+(1-\tilde{\omega})\alpha}{\alpha};
  \]
  \item[(b)] if $[\beta_1+(1-\tilde{\omega})\alpha]^{2}-4\alpha\beta_1> 0$, then
  \[
  0<\tau<\frac{\beta_1+(1-\tilde{\omega})\alpha-\sqrt{[\beta_1+(1-\tilde{\omega})\alpha]^{2}
  -4\alpha\beta_1}}{\alpha}.
  \]
\end{enumerate}
\end{lemma}

Using Lemma \ref{lem:6} and combining the above analyses, we finally obtain the following semi-convergence properties of GSTS iteration method. \begin{theorem}
  Let parameters $\omega_1$, $\omega_2$ and $\tau$ satisfy the conditions of Lemma \ref{lem:7} and matrix $B$, as an approximation of Shur complement $S_M$, be Hermitian positive definite. Then, the GSTS iteration method used for solving singular saddle-point linear system \eqref{01} is semi-convergent.
\end{theorem}

\section{The semi-convergence of GSTS method with $B$ being singular and Hermitian positive semi-definite}\label{sec4}
In this section, we particularly choose matrix $B$ as $B=E^{*}P^{-1}E$, where $P$, as an approximation of $M$, is Hermitian positive definite. Hence, matrix $B$ is singular and has the same null space with Shur complement $S_M$.

In this case, matrix $\mathcal{B}_{c}$ defined in \eqref{13} is singular. We can write matrix $\mathcal{B}(\omega_{1},\omega_{2})$ as
\begin{equation}\label{11}
\begin{split}
\mathcal{B}(\omega_{1},\omega_{2})
&=\left(
     \begin{array}{cc}
       M & 0 \\
       -\omega_{1}E^{*} & B \\
     \end{array}
   \right) \left(
            \begin{array}{cc}
              M^{-1} & 0 \\
              0 & B^{\dag}
            \end{array}
          \right) \left(
                   \begin{array}{cc}
                     M & \omega_{2}E \\
                     0 & B \\
                   \end{array}
                   \right) \\
&=\left(
     \begin{array}{cc}
       M & \omega_{2}E \\
       -\omega_{1}E^{*} & B-\omega_{1}\omega_{2}E^{*}M^{-1}E \\
     \end{array}
   \right),
\end{split}
\end{equation}
where $B^{\dag}$ is the Moore-Penrose inverse of $B$. Since $BB^{\dag}E^{*}=B^{\dag}BE^{*}=E^{*}$ \cite{ZhangShen2013116}, the Moore-Penrose inverse of singular matrix $\mathcal{B}(\omega_{1},\omega_{2})$ has the form of
\begin{equation}\label{12}
\mathcal{B}(\omega_{1},\omega_{2})^{\dag}=\left(
                                            \begin{array}{cc}
                                              M^{-1}-\omega_{1}\omega_{2}M^{-1}EB^{\dag}E^{*}M^{-1} & -\omega_{2}M^{-1}EB^{\dag} \\
                                              \omega_{1}B^{\dag}E^{*}M^{-1} & B^{\dag} \\
                                            \end{array}
                                          \right).
\end{equation}

%Especially when $B$ is taken to be $B:=E^{*}P^{-1}E$, where $P$, it is obvious that the constraint preconditioner $\mathcal{M}$ is singular, making the iteration scheme \eqref{09} is not useful to solve singular saddle-point problem \eqref{01}. So Cao \cite{Cao20081382} gives a new iteration scheme
%\begin{equation}\label{10}
%u^{(k+1)}=u^{(k)}-\mathcal{M}^{\dag}(\mathcal{A}u^{(k)}-f),
%\end{equation}
%where $\mathcal{M}^{\dag}$ is the Moore-Penrose inverse of the singular matrix $\mathcal{M}$. Next we will present two properties of Moore-Penrose inverse that can be used in our later analysis.

%\begin{proposition}\cite{Kucera2011}\label{pro.1}
%Let the singular value decomposition of $A\in\mathbb{C}^{m\times n}$ be
%$$
%A=(U_{1},U_{2})\left(
%               \begin{array}{cc}
%                 \Sigma & 0 \\
%                 0 & 0 \\
%               \end{array}
%             \right)\left(
%                      \begin{array}{c}
%                        V_{1}^{*} \\
%                        V_{2}^{*} \\
%                      \end{array}
%                    \right).
%$$
%Then, the Moore-Penrose inverse of $A$ has the following form
%$$
%A^{\dag}=(V_{1},V_{2})\left(
%               \begin{array}{cc}
%                 \Sigma^{-1} & 0 \\
%                 0 & 0 \\
%               \end{array}
%             \right)\left(
%                      \begin{array}{c}
%                       U_{1}^{*} \\
%                        U_{2}^{*} \\
%                      \end{array}
%                    \right).
%$$
%\end{proposition}

In the following subsections, we analyze the semi-convergence properties of GSTS iteration method according to Lemma \ref{lem:1}.
\subsection{The conditions for $\gamma(\mathcal{G}(\omega_{1},\omega_{2},\tau))<1$}
Based on the singular value decomposition of $E$ defined in \eqref{10}, we have
\begin{eqnarray*}
B=E^{*}P^{-1}E&=&(V_{1},V_{2})\left(
                                    \begin{array}{cc}
                                      \Sigma_{r} & 0 \\
                                      0 & 0 \\
                                    \end{array}
                                  \right)\left(
                                           \begin{array}{c}
                                             U_{1}^{*} \\
                                             U_{2}^{*} \\
                                           \end{array}
                                         \right)P^{-1}(U_{1},U_{2})\left(
                                    \begin{array}{cc}
                                      \Sigma_{r} & 0 \\
                                      0 & 0 \\
                                    \end{array}
                                  \right)\left(
                                           \begin{array}{c}
                                             V_{1}^{*} \\
                                             V_{2}^{*} \\
                                           \end{array}
                                         \right)\\
&=&(V_{1},V_{2})\left(
                                    \begin{array}{cc}
                                      \Sigma_{r}\hat{P}\Sigma_{r} & 0 \\
                                      0 & 0 \\
                                    \end{array}
                                  \right)\left(
                                           \begin{array}{c}
                                             V_{1}^{*} \\
                                             V_{2}^{*} \\
                                           \end{array}
                                         \right),
\end{eqnarray*}
where $\hat{P}=U_{1}^{*}P^{-1}U_{1}$. The Moore-Penrose inverse of $B$ can be written as
\begin{equation*}
B^{\dag}=(V_{1},V_{2})\left(
                        \begin{array}{cc}
                                     (\Sigma_{r}\hat{P}\Sigma_{r})^{-1} & 0 \\
                                      0 & 0 \\
                                    \end{array}
                                  \right)\left(
                                           \begin{array}{c}
                                             V_{1}^{*} \\
                                             V_{2}^{*} \\
                                           \end{array}
                                         \right).
\end{equation*}
Using the unitary matrix $\mathcal{P}$ defined in \eqref{17}, iteration matrix $\mathcal{G}(\omega_{1},\omega_{2},\tau)$ is unitarily similar to the matrix $\hat{\mathcal{G}}(\omega_{1},\omega_{2},\tau)=\mathcal{P}^{*}\mathcal{G}(\omega_{1},\omega_{2},\tau)\mathcal{P}$. Hence, we in this subsection only need to analyze $\gamma(\hat{\mathcal{G}}(\omega_{1},\omega_{2},\tau))<1$.

Define matrices $\hat{M}=U^{*}MU$ and $\hat{S}_P=\Sigma_{r}\hat{P}\Sigma_{r}$, then
\begin{equation*}
\hat{\mathcal{G}}(\omega_{1},\omega_{2},\tau)=\mathcal{P}^{*}\mathcal{G}(\omega_{1},\omega_{2},\tau)\mathcal{P}
=\left(
     \begin{array}{cc}
       \hat{\mathcal{G}}_{1}(\omega_{1},\omega_{2},\tau) & 0 \\
       \hat{\mathcal{L}}(\omega_{1},\omega_{2},\tau) & I_{q-r} \\
     \end{array}
   \right),
\end{equation*}
%\begin{eqnarray*}
%&&\hat{\mathcal{G}}(\omega_{1},\omega_{2},\tau)=\mathcal{P}^{*}\mathcal{G}(\omega_{1},\omega_{2},\tau)\mathcal{P}\\
%&&=\left(
%     \begin{array}{cc}
%       (1-\tau)I_{p}-\tau(1-\omega_{1})\omega_{2}U^{*}M^{-1}EB^{\dag}E^{*}U & -\tau U^{*}M^{-1}EV+ \tau\omega_{1}\omega_{2}U^{*}M^{-1}EB^{\dag}E^{*}M^{-1}EV\\
%       \tau(1-\omega_{1})V^{*}B^{\dag}E^{*}U & I_{q}- \tau\omega_{1}V^{*}B^{\dag}E^{*}M^{-1}EV\\
%     \end{array}
%   \right)\\
%&&=\left(
%     \begin{array}{cc}
%       \hat{\mathcal{G}}_{1}(\omega_{1},\omega_{2},\tau) & 0 \\
%       \hat{\mathcal{L}}(\omega_{1},\omega_{2},\tau) & I_{q-r} \\
%     \end{array}
%   \right)
%\end{eqnarray*}
where
\begin{equation*}
\hat{\mathcal{G}}_{1}(\omega_{1},\omega_{2},\tau)=\left(
                                                    \begin{array}{cc}
       (1-\tau)I_{p}-\tau\omega_{2}(1-\omega_{1})\hat{M}^{-1}E_{r}\hat{S}_P^{-1}E^{*}_{r} & -\tau \hat{M}^{-1}E_{r}+\tau\omega_{1}\omega_{2}\hat{M}^{-1}E_{r}\hat{S}_P^{-1}E^{*}_{r}\hat{M}^{-1}E_{r} \\
                                              \tau(1-\omega_{1})\hat{S}_P^{-1}E^{*}_{r} & I_{q}-\tau\omega_{1}\hat{S}_P^{-1}E^{*}_{r}\hat{M}^{-1}E_{r} \\
                                                    \end{array}
                                                  \right)
\end{equation*}
and
\begin{equation*}
\hat{\mathcal{L}}(\omega_{1},\omega_{2},\tau)=\left(
                                                \begin{array}{cc}
                                                 \tau(1-\omega_{1})V_{2}^{*}V_{1}\hat{S}_P^{-1}E^{*}_{r} & -\tau\omega_{1}V_{2}^{*}V_{1}\hat{S}_P^{-1}E^{*}_{r}\hat{M}^{-1}E_{r} \\
                                                \end{array}
                                              \right).
\end{equation*}
As $E_{r}$ is of full column rank and $\hat{S}_P^{-1}$ is nonsingular, then $\hat{\mathcal{L}}(\omega_{1},\omega_{2},\tau)\neq 0$, so $\gamma(\hat{\mathcal{G}}(\omega_{1},\omega_{2},\tau))<1$ if and only if $\rho(\hat{\mathcal{G}}_{1}(\omega_{1},\omega_{2},\tau))<1$.

Analogously, $\hat{\mathcal{G}}_{1}(\omega_{1},\omega_{2},\tau)$ is the iteration matrix of the GSTS iteration method applied for the nonsingular saddle-point problem
\begin{equation}\label{03}
\hat{\mathcal{A}}\hat{u}=\left(
  \begin{array}{cc}
    \hat{M} & E_{r} \\
    -E^{*}_{r} & 0 \\
  \end{array}
\right)\left(
         \begin{array}{c}
           \hat{u}_{1} \\
           \hat{u}_{2} \\
         \end{array}
       \right)=\left(
                 \begin{array}{c}
                   \hat{f}_{1} \\
                   \hat{f}_{2} \\
                 \end{array}
               \right)=\hat{f},
\end{equation}
where $\hat{M}$ is a Hermitian positive definite matrix, $E_{r}$ is of full column rank and
\begin{equation}\label{04}
\hat{\mathcal{B}}(\omega_{1},\omega_{2})=(\hat{\mathcal{B}}_{c}+\omega_{1}\hat{\mathcal{K}}_{L})\hat{\mathcal{B}}_{c}^{-1}(\hat{\mathcal{B}}_{c}
+\omega_{2}\hat{\mathcal{K}}_{U}),
\end{equation}
with
\[\hat{\mathcal{B}}_{c}=\left(
                        \begin{array}{cc}
                          \hat{M} & 0 \\
                          0 & \hat{S}_P \\
                        \end{array}
                      \right)\]
being Hermitian positive definite since $\hat{S}_P\in \mathbb{C}^{r\times r}$ is Hermitian positive definite.

Under this situation, we denote by
\begin{equation*}
\alpha:=\frac{z^{*}E_{r}^{*}\hat{M}^{-1}E_{r}z}{z^{*}z}\quad \text{and} \quad \beta_{2}:=\frac{z^{*}\hat{S}_Pz}{z^{*}z},
\end{equation*}
By making use of Theorem 3.3 in \cite{Krukier2013}, we derive the following result.
%then there are the following results:
%\begin{theorem}\cite{Krukier2013}
%Let the matrices $\hat{\mathcal{A}}$ and $\hat{\mathcal{B}}(\omega_{1},\omega_{2})$ be defined by \eqref{03} and \eqref{04}, respectively. Assume that $\mu$ is an eigenvalue of the iteration matrix $\hat{\mathcal{G}}_{1}(\omega_{1},\omega_{2},\tau)$ and $(y^{*},z^{*})^{*}\in \mathbb{C}^{p+r}$ is the corresponding eigenvector with $y\in \mathbb{C}^{p}$ and $z\in \mathbb{C}^{r}$. Then $\mu$ satisfies the quadratic equation
%\begin{equation*}
%\beta_{2}\mu^{2}+[\tau(\beta_{2}+\alpha-\tilde{\omega}\alpha)-2\beta_{2}]\mu+\tau^{2}\alpha-\tau(\beta_{2}+\alpha-\tilde{\omega}\alpha)+\beta_{2}=0,
%\end{equation*}
%where $\tilde{\omega}=(\omega_{1}-1)(\omega_{2}-1)$.
%\end{theorem}
\begin{lemma}\label{lem:2}
Denote $\tilde{\omega}=(\omega_{1}-1)(\omega_{2}-1)$. Let matrices $\mathcal{A}$ and $\mathcal{B}(\omega_{1},\omega_{2})$ be defined by \eqref{01} and \eqref{11}, respectively, and $B$=$E^{*}P^{-1}E$. Then $\gamma(\mathcal{G})<1$ holds, provided that the parameters $\omega_{1}$, $\omega_{2}$ satisfy
\[
\tilde{\omega}<\frac{\alpha+\beta_{2}}{\alpha},
\]
and the parameter $\tau$ satisfies
\begin{enumerate}
  \item[(a)] if $[\beta_{2}+(1-\tilde{\omega})\alpha]^{2}-4\alpha\beta_{2}\leq 0$, then
  \[
     0<\tau<\frac{\beta_{2}+(1-\tilde{\omega})\alpha}{\alpha};
  \]
  \item[(b)] if $[\beta_{2}+(1-\tilde{\omega})\alpha]^{2}-4\alpha\beta_{2}> 0$, then
  \[
     0<\tau<\frac{\beta_{2}+(1-\tilde{\omega})\alpha-\sqrt{[\beta_{2}+(1-\tilde{\omega})\alpha]^{2}-4\alpha\beta_{2}}}{\alpha}.
  \]
  \end{enumerate}
\end{lemma}

\subsection{The conditions for $\text{null}(\mathcal{M}^{\dag}\mathcal{A})=\text{null}(\mathcal{A})$}
From iteration scheme \eqref{22}, we have $\mathcal{M}=\tau\mathcal{B}(\omega_{1},\omega_{2})$,
which means
\[\text{null}(\mathcal{M}^{\dag}\mathcal{A})=\text{null}(\mathcal{B}(\omega_{1},\omega_{2})^{\dag}\mathcal{A}).\]
In the following, we only need to verify $\text{null}(\mathcal{B}(\omega_{1},\omega_{2})^{\dag}\mathcal{A})=\text{null}(\mathcal{A})$.
\begin{lemma}\label{lem:3}
Let matrices $\mathcal{A}$ and $\mathcal{B}(\omega_{1},\omega_{2})$ be defined by \eqref{01} and \eqref{11}, respectively, and $B$=$E^{*}P^{-1}E$ with $P$ being Hermitian positive definite. Then $\text{null}(\mathcal{B}(\omega_{1},\omega_{2})^{\dag}\mathcal{A})=\text{null}(\mathcal{A})$.
\end{lemma}
\begin{proof}
Let $x\in\mathbb{C}^{p+q}$ satisfy $\mathcal{B}(\omega_{1},\omega_{2})\mathcal{B}(\omega_{1},\omega_{2})^{\dag}\mathcal{A}x=0$, then
\begin{equation*}
\mathcal{B}(\omega_{1},\omega_{2})^{\dag}\mathcal{A}x=\mathcal{B}(\omega_{1},\omega_{2})^{\dag}
(\mathcal{B}(\omega_{1},\omega_{2})\mathcal{B}(\omega_{1},\omega_{2})^{\dag}\mathcal{A}x)=0.
\end{equation*}
So, we have
\begin{equation*}
\text{null}(\mathcal{B}(\omega_{1},\omega_{2})\mathcal{B}(\omega_{1},\omega_{2})^{\dag}\mathcal{A})\subseteq \text{null}(\mathcal{B}(\omega_{1},\omega_{2})^{\dag}\mathcal{A}).
\end{equation*}
Note that $\text{null}(\mathcal{B}(\omega_{1},\omega_{2})\mathcal{B}(\omega_{1},\omega_{2})^{\dag}\mathcal{A})\supseteq \text{null}(\mathcal{B}(\omega_{1},\omega_{2})^{\dag}\mathcal{A})$ is obvious, we get
\begin{equation}\label{05}
\text{null}(\mathcal{B}(\omega_{1},\omega_{2})\mathcal{B}(\omega_{1},\omega_{2})^{\dag}\mathcal{A})= \text{null}(\mathcal{B}(\omega_{1},\omega_{2})^{\dag}\mathcal{A}).
\end{equation}
%Note that $\mathcal{B}(\omega_{1},\omega_{2})^{\dag}$ is the Moore-Penrose inverse of $\mathcal{B}(\omega_{1},\omega_{2})$, it follows that
%\begin{equation*}
%\mathcal{B}(\omega_{1},\omega_{2})\mathcal{B}(\omega_{1},\omega_{2})^{\dag}=\left(
%                                                                              \begin{array}{cc}
%                                                                                I_{p} & 0 \\
%                                                                               0 & B^{\dag}B \\
%                                                                              \end{array}
%                                                                            \right).
%\end{equation*}
%Hence
Simple calculation gives
\begin{equation}\label{06}
\mathcal{B}(\omega_{1},\omega_{2})\mathcal{B}(\omega_{1},\omega_{2})^{\dag}\mathcal{A}=\left(
                                                                              \begin{array}{cc}
                                                                                I_{p} & 0 \\
                                                                               0 & B^{\dag}B \\
                                                                              \end{array}
                                                                            \right)\left(
                                                                                     \begin{array}{cc}
                                                                                       M & E \\
                                                                                       -E^{*} & 0 \\
                                                                                     \end{array}
                                                                                   \right)=\left(
                                                                                     \begin{array}{cc}
                                                                                       M & E \\
                                                                                       -E^{*} & 0 \\
                                                                                     \end{array}
                                                                                   \right)=\mathcal{A}.
\end{equation}
Hence, using \eqref{05} and \eqref{06}, we finally obtain that $\text{null}(\mathcal{B}(\omega_{1},\omega_{2})^{\dag}\mathcal{A})=\text{null}(\mathcal{A})$.
\end{proof}
\subsection{the conditions for index$(I-\mathcal{G}(\omega_{1},\omega_{2},\tau))=1$}
\begin{lemma}\label{lem:4}
Let matrices $\mathcal{A}$ and $\mathcal{B}(\omega_{1},\omega_{2})$ be defined by \eqref{01} and \eqref{11}, respectively, and $B$=$E^{*}P^{-1}E$, with $P$ being Hermitian positive definite. Then $\text{index}(I-\mathcal{G}(\omega_{1},\omega_{2},\tau))=1$ or equivalently, \begin{equation}\label{29}
\text{rank}(I-\mathcal{G}(\omega_{1},\omega_{2},\tau))=\text{rank}((I-\mathcal{G}(\omega_{1},\omega_{2},\tau))^{2}).
\end{equation}
\end{lemma}
\begin{proof}
Since $I-\mathcal{G}(\omega_{1},\omega_{2},\tau)=\tau\mathcal{B}(\omega_{1},\omega_{2})^{\dag}\mathcal{A}$, the equality \eqref{29} holds if
\[
\text{null}((\mathcal{B}(\omega_{1},\omega_{2})^{\dag}\mathcal{A})^{2})=\text{null}(\mathcal{B}(\omega_{1},\omega_{2})^{\dag}\mathcal{A}).
\]
Since $\text{null}((\mathcal{B}(\omega_{1},\omega_{2})^{\dag}\mathcal{A})^{2})\supseteq \text{null}(\mathcal{B}(\omega_{1},\omega_{2})^{\dag}\mathcal{A})$ is obvious, we only need to prove \[\text{null}((\mathcal{B}(\omega_{1},\omega_{2})^{\dag}\mathcal{A})^{2})\subseteq \text{null}(\mathcal{B}(\omega_{1},\omega_{2})^{\dag}\mathcal{A}).\]

Suppose that $x=(x_{1}^{*},x_{2}^{*})^{*}\in\mathbb{C}^{p+q}$ satisfies $(\mathcal{B}(\omega_{1},\omega_{2})^{\dag}\mathcal{A})^{2}x=0$, we have
\begin{eqnarray*}
\mathcal{B}(\omega_{1},\omega_{2})^{\dag}\mathcal{A}x&=&\left(
                                                          \begin{array}{cc}
                                                            I_{p}+\omega_{2}(1-\omega_{2})M^{-1}EB^{\dag}E^{*} & M^{-1}E-\omega_{1}\omega_{2}M^{-1}EB^{\dag}E^{*}M^{-1}E \\
                                                            (\omega_{1}-1)B^{\dag}E^{*} & \omega_{1}B^{\dag}E^{*}M^{-1}E \\
                                                          \end{array}
                                                        \right)\left(
                                                                 \begin{array}{c}
                                                                   x_{1} \\
                                                                   x_{2} \\
                                                                 \end{array}
                                                               \right)\\
&=&\left(
                                                          \begin{array}{cc}
                                                            (I_{p}+\omega_{2}(1-\omega_{2})M^{-1}EB^{\dag}E^{*})x_{1}+(M^{-1}E-\omega_{1}\omega_{2}M^{-1}EB^{\dag}E^{*}M^{-1}E)x_{2} \\
                                                            (\omega_{1}-1)B^{\dag}E^{*}x_{1}+(\omega_{1}B^{\dag}E^{*}M^{-1}E)x_{2} \\
                                                          \end{array}
                                                        \right)\\
&=&\left(
     \begin{array}{c}
       y_{1} \\
       y_{2} \\
     \end{array}
   \right)\equiv y.
\end{eqnarray*}
In the following, we only need to prove $\mathcal{B}(\omega_{1},\omega_{2})^{\dag}\mathcal{A}x=y=0$. Owing to $\text{null}(\mathcal{B}(\omega_{1},\omega_{2})^{\dag}\mathcal{A})=\text{null}(\mathcal{A})$ and \[\mathcal{B}(\omega_{1},\omega_{2})^{\dag}\mathcal{A}y=(\mathcal{B}(\omega_{1},\omega_{2})^{\dag}\mathcal{A})^{2}x=0,\]
we have $\mathcal{A}y=0$, i.e.,
\begin{equation}\label{07}
My_{1}+Ey_{2}=0 \quad \text{and} \quad -E^{*}y_{1}=0.
\end{equation}
Since $M$ is nonsingular, solving $y_{1}$ from the first equality of \eqref{07} and taking into the second equality, we have $E^{*}M^{-1}Ey_{2}=0$, which means
\begin{equation*}
(Ey_{2})^{*}M^{-1}(Ey_{2})=y_{2}(E^{*}M^{-1}Ey_{2})=0.
\end{equation*}
Owing to the positive definiteness of matrix $M^{-1}$, we can obtain that $Ey_{2}=0$. Hence, using the first equality of \eqref{07} gives $y_{1}=0$.

Using $Ey_{2}=0$ and $B^{\dag}BE^{*}=E^{*}$\cite{ZhangShen2013116}, we have
\begin{eqnarray*}
y_{2}&=&(\omega_{1}-1)B^{\dag}E^{*}x_{1}+(\omega_{1}B^{\dag}E^{*}M^{-1}E)x_{2}\\
&=&B^{\dag}E^{*}P^{-1}[(\omega_{1}-1)EB^{\dag}E^{*}x_{1}+(\omega_{1}EB^{\dag}E^{*}M^{-1}E)x_{2}]\\
&=&B^{\dag}E^{*}P^{-1}(Ey_{2})=0
\end{eqnarray*}
Finally, we obtain $y=(y_{1}^{*},y_{2}^{*})^{*}=0$, so $\text{null}((\mathcal{B}(\omega_{1},\omega_{2})^{\dag}\mathcal{A})^{2})=\text{null}(\mathcal{B}(\omega_{1},\omega_{2})^{\dag}\mathcal{A})$.
\end{proof}

Using Lemmas \ref{lem:1} and \ref{lem:2}-\ref{lem:4}, we obtain the semi-convergence property of GSTS iteration method for singular saddle-point linear systems.
\begin{theorem}
  Let parameters $\omega_1$, $\omega_2$ and $\tau$ satisfy the conditions of Lemma \ref{lem:2} and matrix $B=E^{*}P^{-1}E$ with $P$ being a Hermitian positive definite approximation of matrix $M$. Then, the GSTS iteration method used for solving singular saddle-point linear system \eqref{01} is semi-convergent.
\end{theorem}

In GSTS iteration method, we can particularly choose $P=M$ since $P$ is an approximation of $M$. In this case, we have $B=E^{*}M^{-1}E=S_M$. Simple calculation gives
\begin{equation*}
\alpha=\frac{z^{*}E_{r}^{*}\hat{M}^{-1}E_{r}z}{z^{*}z}
=\frac{z^{*}(\Sigma_{r},0)U^{*}M^{-1}U\left(
                \begin{array}{c}
                  \Sigma_{r} \\
                  0 \\
                \end{array}
              \right)z}{z^{*}z}=\frac{z^{*}\Sigma_{r}U_{1}^{*}M^{-1}U_{1}\Sigma_{r}z}{z^{*}z},
\end{equation*}
and
\begin{equation*}
\beta_{2}=\frac{z^{*}\hat{S}_Pz}{z^{*}z}=\frac{z^{*}\Sigma_{r}\hat{P}^{-1}\Sigma_{r}z}{z^{*}z}=\frac{z^{*}\Sigma_{r}U_{1}^{*}P^{-1}U_{1}\Sigma_{r}z}{z^{*}z}
=\frac{z^{*}\Sigma_{r}U_{1}^{*}M^{-1}U_{1}\Sigma_{r}z}{z^{*}z}.
\end{equation*}
Obviously, under this assumption, we have $\alpha=\beta_{2}$. The convergence property of the particular GSTS method becomes
\begin{corollary}
Denote $\tilde{\omega}=(\omega_{1}-1)(\omega_{2}-1)$. Let matrices $\mathcal{A}$ and $\mathcal{B}(\omega_{1},\omega_{2})$ be defined by \eqref{01} and \eqref{11}, respectively, and $B$=$E^{*}M^{-1}E$. Then, the GSTS iteration method is semi-convergent, provided that the parameters $\omega_{1}$, $\omega_{2}$ satisfy $\tilde{\omega}<2$ and the parameter $\tau$ satisfies
\begin{enumerate}
  \item[(a)] if $0\leq\tilde{\omega}<2$, then $0<\tau<2-\tilde{\omega}$;
  \item[(b)] if $\tilde{\omega}<0$, then $0<\tau<2-\tilde{\omega}-\sqrt{\tilde{\omega}(\tilde{\omega}-4)}$.
\end{enumerate}
\end{corollary}

\section{Numerical results}\label{sec5}
In this section, we assess the feasibility and robustness of the GSTS iteration methods for solving the singular saddle-point problems \eqref{01}. In addition, the preconditioning effects of the GSTS preconditioners for GMRES(10) and QMR methods will also be tested.

Consider the Stokes equations of the following form
\begin{equation}\label{34}
    \left\{
    \begin{split}
      &-\nu \Delta \textbf{u}+\nabla p = \textbf{f},\quad \text{in }\Omega,\\
      &-\nabla\cdot \textbf{u}=0,\quad \text{in }\Omega,
    \end{split}\right.
\end{equation}
where $\Omega$ is an open bounded domain in $\mathbb{R}^2$, vector $\textbf{u}$ represents the velocity in $\Omega$, function $p$ represents pressure, and
the scalar $\nu>0$ is the viscosity constant. The boundary conditions are $\textbf{u} = (0,0)^T$ on the three
fixed walls $(x = 0, y = 0, x = 1)$, and $\textbf{u} = (1, 0)^T$ on the moving wall $(y = 1)$.

Dividing $\Omega$ into a uniform $l\times l$ grid with mesh size $h=1/l$ and discretizing \eqref{34} by the "marker and cell" (MAC) finite difference scheme \cite{HW19652182,Elman19991299}, the singular saddle-point system \eqref{01} is obtained, where
\[
M=\nu\left( \begin{array}{cc} A_1 & 0 \\ 0 & A_2 \end{array} \right)\in\mathbb{R}^{2l(l-1)\times 2l(l-1)},\quad E^*=(E_1,E_2)\in\mathbb{R}^{l^2\times 2l(l-1)}.
\]
The coefficient matrix $\mathcal{A}$ of \eqref{01} has the following properties: $M$ is symmetric and positive definite, $\text{rank}(E) = l^2-1$, thus $\mathcal{A}$ is singular.

Based on the different choices of matrix $P$
and parameters $\omega_1$, $\omega_2$ and $\tau$, we test five cases of the GSTS iteration method listed in Table \ref{tab:01}. In order to reduce the complexity for finding the experimental optimal values of $\omega_1$, $\omega_2$ and $\tau$, we particularly choose $\omega_{1}=\tau$. The last two cases of GSTS iteration method reduce to the generalized successive overrelaxation (GSOR) methods discussed in \cite{ZBY2009808}. Here, $\omega_{\exp}$ and $\tau_{\exp}$ in GSTS methods denote the experimental optimal values of the iteration parameters $\omega_{1}$ and $\tau$, respectively, while $\omega_{\mathrm{opt}}$ and $\nu_{\mathrm{opt}}$ in GSOR denote the theoretical optimal values; see Theorem 4.1 in \cite{ZBY2009808}.
\begin{table}[h]
\caption{The five cases of GSTS iteration method} \label{tab:01}
\begin{center}
\begin{tabular}{*{3}{l}}
\hline\noalign{\smallskip}
Methods& Preconditioning matrix $B$ &Parameters\\
\hline\noalign{\smallskip}
GSTS I   &$E^{*}M^{-1}E$ & $\omega_1=\tau=\tau_{\exp}$, $\omega_2=\omega_{\exp}$\\
GSTS II  &$I+E^{*}P^{-1}E$, with $P=\text{diag}(M)$ & $\omega_1=\tau=\tau_{\exp}$, $\omega_2=\omega_{\exp}$\\
GSTS III &$I+E^{*}P^{-1}E$, with $P=\text{tridiag}(M)$  & $\omega_1=\tau=\tau_{\exp}$, $\omega_2=\omega_{\exp}$\\
GSOR I  &$\frac{\tau}{\nu}(I+E^{*}P^{-1}E)$, with $P=\text{diag}(M)$ & $\omega_1=\tau=\omega_{\mathrm{opt}}$, $\omega_2=0$, $\nu=\tau_{\mathrm{opt}}$\\
GSOR II &$\frac{\tau}{\nu}(I+E^{*}P^{-1}E)$, with $P=\text{tridiag}(M)$ & $\omega_1=\tau=\omega_{\mathrm{opt}}$, $\omega_2=0$, $\nu=\tau_{\mathrm{opt}}$\\
\noalign{\smallskip}\hline
\end{tabular}
\end{center}
\end{table}

In actual computations, we choose $l=25$ and the right-hand-side vector $f\in \mathbb{R}^{3l^{2}-2l}$ such that the exact solution of \eqref{01} is $u^{*}=(1,2,\cdots,3l^{2}-2l)^{T}\in \mathbb{R}^{3l^{2}-2l}$. The iteration methods are started from zero vector and terminated once the current iterate $x^{(k)}$ satisfies
\begin{equation}\label{40}
\text{RES}=\sqrt{\frac{\|f_{1}-Mu^{(k)}_{1}-Eu^{(k)}_{2}\|^{2}_{2}+\|f_{2}+E^*u^{(k)}_{1}\|^{2}_{2}}{\|f_{1}\|^{2}_{2}+\|f_{2}\|^{2}_{2}}}< 10^{-6}.
\end{equation}
In addition, all codes were run in MATLAB [version 7.10.0.499 (R2010a)] in double precision and all experiments were performed on a personal computer with 3.10GHz central processing unit [Intel(R) Core(TM) Duo i5-2400] and 3.16G memory.

In Table \ref{tab:02}, we present the numerical results including iteration steps (denoted as IT),
elapsed CPU time in seconds (denoted as CPU) and relative residuals (denoted as RES) of the GSTS and
GSOR iteration methods listed in Table \ref{tab:01} and GMRES method. From the numerical results we
see that all the testing methods can converge to the approximate solutions. The five cases of GSTS
method perform better than GMRES method in iteration steps and CPU times. During the five cases
of GSTS method, the second and third cases, i.e., GSTS II and GSTS III, always outperform the fourth
and fifth cases, i.e., GSOR I and GSOR II methods, respectively, especially for the elapsed CPU time.
In GSTS I, we choose $B$ being a singular matrix. Comparing with GMRES and other four cases of GSTS method, GSTS I uses the least iteration number and CPU time to achieve the stop criterion.
% It should be mentioned that the smaller of the parameter $\nu$, the stronger of the skew-Hermitian part of the saddle-point matrix. From table \ref{tab:02}, we also see when $\nu$ decreases, the convergence properties of all testing methods also perform very well.
\begin{table}[!h]
\caption{Numerical results of GSTS, GSOR and GMRES iteration methods} \label{tab:02}
\begin{center}
\begin{tabular}{rl*{5}{c}}
\hline\noalign{\smallskip}
&Method&$\omega_{\exp}(\omega_{\mathrm{opt}})$&$\tau_{\exp}(\tau_{\mathrm{opt}})$&IT&CPU&RES\\
\hline\noalign{\smallskip}
$\nu=1$      & GSTS I   & 0.98 & 1.01 & 3 & 0.0156 & 8.3841e-7 \\
             & GSTS II   & 0.99 & 1.03 & 12 & 0.0312 & 9.7573e-7 \\
             & GSTS III  & 0.97 & 1.02 & 13 & 0.0468 & 9.6413e-7 \\
             & GSOR I   & 0.89 & 2.07 & 14 & 0.1248 & 4.0305e-7 \\
             & GSOR II  & 0.89 & 2.11 & 14 & 0.1872 & 4.0232e-7\\
             & GMRES    & $-$ & $-$ & 182 & 1.3104 & 9.9927e-7\\
$\nu=0.01$   & GSTS I   & 0.99 & 1.00 & 3 & 0.0468 & 6.6282e-7 \\
             & GSTS II  & 0.01 & 0.30 & 73 & 0.1872 & 9.7501e-7 \\
             & GSTS III  & 0.02 & 0.34 & 63 & 0.1872 & 9.6153e-7 \\
             & GSOR I   & 0.38 & 0.24 & 68 & 0.4524 & 8.1165e-7 \\
             & GSOR II  & 0.43 & 0.28 & 58 & 0.5304 & 8.5245e-7 \\
             & GMRES    & $-$ & $-$ & 404 & 6.5988 & 9.7660e-7\\
$\nu=0.0001$ & GSTS I   & 0.98 & 1.00 & 3 & 0.0312 & 1.4299e-12 \\
             & GSTS II  & 0.01 & 0.17 & 115 & 0.3901 & 9.5478e-7 \\
             & GSTS III  & 0.01 & 0.24 & 110 & 0.3276 & 9.4855e-7 \\
             & GSOR I   & 0.24 & 0.14 & 164 & 0.7176 & 9.1062e-7 \\
             & GSOR II  & 0.32 & 0.20 & 118 & 0.8580 & 9.8351e-7 \\
             & GMRES    & $-$ & $-$ & 637 & 15.6004 & 9.7598e-7\\
\noalign{\smallskip}\hline
\end{tabular}
\end{center}
\end{table}

In addition to using GSTS as an iteration solver, we also use it to precondition GMRES method. The preconditioning effects of the five cases of GSTS method are compared with those of the Hermitian and skew-Hermitian splitting (HSS) preconditioner \cite{BGL20051,BGN2002603,BG200520} and the constraint preconditioner \cite{KGW20001300,ZhangWei2010139,Cao20081382,ZhangShen2013116}. Here, the non-singular constraint preconditioner (CP) is of the form
\[\mathcal{P}_c=\left( \begin{array}{cc} P & E \\ -E^* & I \end{array} \right),\]
where $P$ is an approximate matrix of $M$ and $I$ is an identity matrix. We name the constraint preconditioners $\mathcal{P}_c$ with $P=\text{diag}(M)$ and $P=\text{tridiag}(M)$, respectively, as CP I and CP II preconditioners. The HSS preconditioner is of the form
\[\mathcal{P}_{h}=(\alpha I+\mathcal{A}_{H})(\alpha I+\mathcal{A}_{S}),\]
where $\alpha>0$ is a constant, $\mathcal{A}_{H}$ and $\mathcal{A}_{S}$ are defined in \eqref{32}. In the implementation, $\alpha$ is chosen to be the experimental optimal value.
\begin{table}[!h]
\caption{Numerical results of GSTS and GSOR preconditioned GMRES methods} \label{tab:03}
\begin{center}
\begin{tabular}{l*{9}{c}}
\hline\noalign{\smallskip}
&&\multicolumn{2}{l}{$\nu=1$}&&\multicolumn{2}{l}{$\nu=0.01$}&&\multicolumn{2}{l}{$\nu=0.0001$}\\
\cline{3-4}\cline{6-7}\cline{9-10}\noalign{\smallskip}
Method&&IT&CPU&&IT&CPU&&IT&CPU\\
\hline\noalign{\smallskip}
GSTS I   && 14 & 0.0780 && 11 & 0.1560 && 3 & 0.1248 \\
GSTS II   && 17 & 0.0624 && 18 & 0.1872 && 4 & 0.1248 \\
GSTS III && 17 & 0.0757 && 16 & 0.2028 && 3 & 0.1404 \\
GSOR I   && 26 & 0.0780 && 27 & 0.6084 && 8 & 0.1560 \\
GSOR II  && 26 & 0.0936 && 22 & 0.7332 && 6 & 0.1716 \\
CP I    && 34 & 0.0793 && 29 & 0.3225 && 41 & 0.4209 \\
CP II   && 34 & 0.0880 && 28 & 0.3573 && 41 & 0.4370 \\
HSS      && 21 & 0.0816 && 22 & 0.4212 && 19 & 0.4056 \\
\noalign{\smallskip}\hline
\end{tabular}
\end{center}
\end{table}

In Table \ref{tab:03}, we list the iteration numbers and CPU times of the preconditioned GMRES methods used for solving singular saddle-point linear system \eqref{01}. From the numerical results, we see that the superiorities of GSTS preconditioners, comparing with the constraint and HSS preconditioners, become more and more evident with the decrease of parameter $\nu$. This may be because the smaller of the parameter $\nu$, the stronger of the skew-Hermitian part of the saddle-point matrix. In addition, we can also find that the preconditioning effect of singular GSTS I preconditioner is the best one during the eight preconditioners listed in Table \ref{tab:03}.

Thus, we can conclude that the GSTS iteration methods, no matter used as solvers or as preconditioners for GMRES method, are always feasible and effective for solving singular saddle-point linear systems.

\section{Conclusion}\label{sec6}
In this work, we used the GSTS iteration methods to solve \emph{singular} saddle-point linear system \eqref{01}. For each of the two choices of preconditioning matrix $B$, the semi-convergence conditions of GSTS iteration method were derived. Numerical results verified the effectiveness of the GSTS method both used as a solver and as a preconditioner for the GMRES method.

However, the GSTS method involves three iteration parameters $\omega_1$, $\omega_2$ and $\tau$. The choices of these parameters were not discussed in this work since it is a very difficult and complicated task. Considering that the efficiency of GSTS method largely depends on the values of these parameters, how to determine efficient and easy calculated parameters should be a direction of future research.

\bibliographystyle{model1-num-names}

\bibliography{References}

\end{document}